\documentclass[12pt]{amsart}
\usepackage[colorlinks=true,citecolor=blue,linkcolor=blue, backref=page]{hyperref}
\usepackage{amssymb,verbatim,amscd,amsmath,graphicx,nccmath}
\usepackage{enumerate}
\usepackage{graphicx}
\usepackage{caption}
\usepackage{subcaption}
\usepackage{pdfsync}
\usepackage{color,colortbl}
\usepackage{fullpage}
\usepackage{bbm}
\usepackage{comment}
\usepackage{multirow}
\usepackage{enumitem}
\usepackage{comment}
\usepackage{cleveref}
\numberwithin{equation}{section}
\newtheorem{introthm*}{Main Result}
\newtheorem{theorem}{Theorem}[section]
\newtheorem{lemma}[theorem]{Lemma}
\newtheorem{proposition}[theorem]{Proposition}
\newtheorem{corollary}[theorem]{Corollary}

\theoremstyle{definition}
\newtheorem{definition}[theorem]{Definition}

\newtheorem{def-prop}[theorem]{Definition-Proposition}

\newtheorem{remark}[theorem]{Remark}
\newtheorem{example}[theorem]{Example}

\newtheorem*{acknowledgment}{Acknowledgments}

\newtheorem{question}[theorem]{Question}

\DeclareMathOperator{\Ass}{Ass}
\DeclareMathOperator{\Min}{Min}

\DeclareMathOperator{\Spec}{Spec}

\newcommand{\CC}{{\mathbb C}}

\newcommand{\ZZ}{{\mathbb Z}}
\newcommand{\NN}{{\mathbb N}}
\newcommand{\QQ}{{\mathbb Q}}

\newcommand{\kk}{{\mathbbm k}}

\def\pp{{\frak p}}
\def\qq{{\frak q}}

\def\1{{\bf 1}}
\def\0{{\bf 0}}

\allowdisplaybreaks

\begin{document}
	
	\title{Symbolic Powers via Extension}
	
	\author{Sankhaneel Bisui}
	\address{SRM University AP \\ Department of Mathematics \\
		Amaravati, AP 522502, India}
	\email{sankhaneel.b@srmap.edu.in, sbisui@tulane.edu} 
	
	\author{Haoxi Hu}
	\address{Tulane University, Department of Mathematics and Statistics,
		6823 St. Charles Avenue
		New Orleans, LA 70118, USA}
	
	\email{hhu5@tulane.edu}

	
	\begin{abstract}
	This article investigates under which conditions symbolic powers of the extension of an ideal are the same as the extension of the symbolic powers. As an application, we prove formulas for the resurgence of sum of two homogeneous ideals in finitely generated graded $\kk$-algebras which are domains, where $\kk$ is algebraically closed. Initially, these were known for homogeneous ideals in polynomial rings. 
	\end{abstract}
	\maketitle
	\section{Introduction}
For an ideal $I$ in a Noethrian ring $R$, the $m$-th symbolic power is defined as follows: 
$$I^{(m)}=\bigcap_{\pp\in \Ass(I)} I^mR_\pp \cap R .$$ 

The symbolic powers are very well-studied objects for commutative algebra and algebraic geometry, especially for the multiplicity theory. We are interested in the behavior of symbolic powers under ring homomorphisms. If $\phi: A\to B$ is a homomorphism of rings, then for an ideal $I$ in $A$, we use $IB$ to denote the \emph{extension} of the ideal $I$ in $B$, and for an ideal $J$ in $B$, we use $J \cap A$ to denote the \emph{contraction} of the ideal $J$ in $A$. In this article, we primarily focus on the following question. 

\begin{question}\label{question: symbolic power extension}
	Let $A$ and $B$ be Noetherian rings, and $I$ be an ideal of $A$. Let $\phi: A\to B$ be a ring homomorphism. 
	When is $(IB)^{(n)} = I^{(n)}B$, where $m \geqslant 1$? 
\end{question}

If the map $\phi$ is \emph{faithfully flat}, then the Question \ref{question: symbolic power extension} has been well studied, for examples check, \cite[Theorem 3.6]{GHMNmatroid}, \cite[Theorem 4]{Akessesh}, and \cite[Proposition 2.1]{walker2018uniform}. Our result generalizes from \emph{faithfully flat} map to \emph{flat} map, which is the following:

\begin{introthm*} [\Cref{theorem: symbolic powers under flat extension}]
	Let $A$ and $B$ be Noetherian rings  and $\phi: A\to B$ be a flat morphism. Assume that $I\subset A$ is a proper ideal such that $\Ass^*(IB)=\{\pp B~|~ \text{ for all } \pp\in \Ass^*(I) \}$, and every associated prime of $I$ is a contraction of some ideal in $B$. Then $(IB)^{(n)} = I^{(n)}B $, for all $n \geqslant 1$. 
\end{introthm*}  
 We also give examples that satisfy assumptions in Theorem \ref{theorem: symbolic powers under flat extension} but not under the \emph{faithfully flat} map, such as Example \ref{example: flat not faithfully flat 1} and Example \ref{example: flat not faithfully flat 2}. Moreover, In Example \ref{example: contraction condition}, we show that the condition that associated primes are contractions is necessary.

One of the central topics of symbolic powers is the containment between the symbolic and ordinary powers of ideals.  The containment problem asks to find $m$ and $r$ for which $I^{(m)} \subset I^r$. In \cite{EinLazarsfeldSmith, HochsterHuneke, MaScwede}, it has been proved that if $I$ is a proper ideal of big height $h$ (the largest height of associated primes of $I$) in a regular ring containing a field, then $I^{(rh)}\subset I^r$ for all $r \geqslant 1$. After this celebrated result, one can keep investigating exact numbers for containments. The resurgence, defined by Bocci and Harbourne \cite{bocci2010comparing}, and the asymptotic resurgence, defined by Guardo, Harbourne, and Van Tuyl \cite{guardo2013asymptotic} measure containment in more refined settings.  

If $I$ is a proper ideal in a Notherian ring $R$, $\rho(I)$ and the \emph{asymptotic resurgence}, $\rho_a(I)$ are defined as follows: 
$$
\rho(I) =  \sup \left\{ \dfrac{m}{r} ~|~ I^{(m)} \not\subseteq I^r, m\geqslant 1, r\geqslant 1 \right\}, $$ $$ {\rho_a}(I) = \sup \left\{ \frac{m}{r} ~|~ I^{(mt)} \not\subseteq {I^{rt}}, m\geqslant 1, r\geqslant 1, \text{ and } t \gg 1 \right\}. 
$$

Motivated by Question \ref{question: symbolic power extension} and \emph{resurgence numbers},  we study the following question. 
\begin{question}\label{question: resurgence extension}
	Let $A$ and $B$ be Notherian and regular $\kk$-algebras,  $I$ be an ideal of $A$, and $J$ be an ideal of $B$. What can be said about $\rho(I+J)$ and $\rho_a(I+J)$, where $I+J=IR+JR$ and $R=A \otimes_{\kk} B$?   
\end{question}

The resurgence and the asymptotic resurgence of the sum of ideals $I\subset \kk[\underline{x}]$ and $J\subset \kk[\underline{y}]$, where $\kk[\underline{x}]$ and $\kk[\underline{y}]$ are polynomial rings, are studied in \cite{fiberprojective, sharpbound}. 
As an application of \Cref{theorem: symbolic powers under flat extension}, we extend those results ideals in standard graded finitely generated $\kk$-algebras which are domains, and $\kk$ is algebraically closed.
\begin{introthm*} Let $I \subset A $ and $J \subset B$ be homogeneous ideals in finitely generated graded regular $\kk$-algebra domains $A$ and $B$, where $\kk$ is algebraically closed.  Set $I+J=IR+JR$, where $R=A \otimes_{\kk} B$. Then the following hold.  
	\begin{enumerate}
		
		\item (\Cref{theorem: asymptotic resurgence of sum of ideals})  $	\rho_{a}(I + J) = {\max}\{ \rho_a(I),  \rho_a(J) \} =\max\{\rho_a(IR),  \rho_a(JR)  \}$, and also\\
		$\rho_{a}(I + J) = \overline{\rho}_{a}(I + J) = \overline{\rho}(I + J)$, where $\bar{\rho}$ and $\bar{\rho_{\alpha}}$ are resurgence (asymptotic resurgence respectively) of integral closure.

		\item  (\Cref{theorem: resurgence bound of sum of ideals}) $ \displaystyle 
	\max \{ \rho(I), \, \rho(J) \} \leqslant \rho(I + J) \leqslant \max \left\{ \rho(I), \, \rho(J), \, \frac{2 (\rho(I) + \rho(J))}{3} \right\}.$  \\
	If $\max \{ \rho(I), \, \rho(J) \} \geqslant 2 \min \{ \rho(I), \, \rho(J) \}$, then $\rho(I + J) = \max \{ \rho(I), \, \rho(J) \}$.		
	\end{enumerate}	
\end{introthm*}

The paper is outlined as follows. In \Cref{section: Resurgence via prime extension}, we prove our main result regarding the relation symbolic powers of an ideal with its extension, and we also study \emph{prime extension property} and its applications. In \Cref{section: resurgence of sum of ideals}, we apply results from \Cref{section: Resurgence via prime extension} and we extend results regarding the sum of ideals in polynomial rings to finitely generated algebra domains over algebraically closed fields. 

\begin{acknowledgment}
The authors are thankful to Tài Huy Hà for suggesting the problem and valuable comments. We are also thankful to Souvik Dey and Hop D. Nguyen for some helpful discussions.
\end{acknowledgment}

\section{Symbolic Powers and Resurgence via flat extension}\label{section: Resurgence via prime extension}
	In this section we will show our main theorem (Theorem \ref{theorem: symbolic powers under flat extension}). Before proceeding to the main theorem we provide some useful lemmas. The first lemma is very well-known and we state it without proof. 
	\begin{lemma}\label{lemma: extension of an ideal}
		Let $M$ be a $R$-module where $R$ is a commutative ring, and $I$ be an ideal of $R$, then $(R/I) \otimes_R M \cong M/IM$. Furthermore, if $M$ is flat over $A$, we also have $I \otimes_R M \cong IM$.
	\end{lemma}
	
	
The next result is also a standard one. Since we could not find a resource, we provide a complete proof. 
	
	\begin{lemma}\label{lemma: tensor of localization}
	If a flat ring homomorphism $\phi: A\to B$ extends a prime $\mathfrak{p}$ to a prime $\mathfrak{p}B$, and $\mathfrak{p}$ is a contraction of some ideals in $B$, then $B_{\mathfrak{p}B} \cong A_{\mathfrak{p}} \otimes_A B_{\mathfrak{p}B}$ as $A$-algebra.
	\end{lemma}
	
	\begin{proof}
	Applying Lemma \ref{lemma: extension of an ideal} to $B_{\mathfrak{p}B}$, we have $B_{\mathfrak{p}B} \cong ( (A \otimes_A B) \backslash (\mathfrak{p}B \otimes_A B) )^{-1} (A \otimes_A B)$. Our goal is to show $( (A \otimes_A B) \backslash (\mathfrak{p}B \otimes_A B) )^{-1} (A \otimes_A B) \cong A_{\mathfrak{p}} \otimes_A B_{\mathfrak{p}B}$.
	
	Define $$\psi_1 : ( (A \otimes_A B) \backslash (\mathfrak{p}B \otimes_A B) )^{-1} (A \otimes_A B) \rightarrow A_{\mathfrak{p}} \otimes_A B_{\mathfrak{p}B} \text{ by } \psi_1\left(\frac{a \otimes b}{c \otimes d}\right) = \frac{a}{c} \otimes \frac{b}{d}.   $$ 
	
	Similarly define  $$ \psi_2 : A_{\mathfrak{p}} \otimes_A B_{\mathfrak{p}B} \rightarrow ( (A \otimes_A B) \backslash (\mathfrak{p}B \otimes_A B) )^{-1} (A \otimes_A B) \text{ by } \psi_2\left(\frac{a}{c} \otimes \frac{b}{d}\right) = \frac{a \otimes b}{c \otimes d}.$$ 
	
	First we need to show both maps are well-defined. 
	
	Since $c \otimes d \notin \mathfrak{p} \otimes_A B$, then $c \notin \mathfrak{p}$. We also can consider $c \otimes d = 1 \otimes cd \notin 1 \otimes_A \mathfrak{p}B$, so $cd \notin \mathfrak{p}B$. We also know that $c \notin \mathfrak{p}$, this forces $d \notin \mathfrak{p}B$ since $\mathfrak{p}B$ is prime. Therefore $\psi_1$ is well-defined.
	
 We prove that	$\psi_2$ is also well-defined. Indeed, $c \notin \mathfrak{p}$ and $d \notin \mathfrak{p}B$, then $\phi(c)d \notin \mathfrak{p}B$ because $\phi(A) \cap \mathfrak{p}B = \phi(\mathfrak{p})$. To see this, let $\mathfrak{p} = \mathfrak{q}\cap A$ for an ideal $\mathfrak{q}$ in $B$. From this we  observe that $\mathfrak{p}B\cap A  = \left( \left( \mathfrak{q}\cap A\right) B\right) \cap A  = \mathfrak{q}\cap A=\pp$. This implies that $\phi(A) \cap \mathfrak{p}B = \phi(\mathfrak{p})$. Note that $\phi(A) \cap \mathfrak{p}B \supset \phi(\mathfrak{p})$. Let $y \in \phi(A)\cap \pp B$. Then $y=\phi(x)$ for some $x\in A$. We show that $x \in \pp$. Now since $y \in \pp B$ then $x\in \pp B\cap A=\pp$. 
 
Notice that $\psi_1$ and $\psi_2$ are inverses to each other. Thereofre we just need to show that $\psi_1$ is an $A$-algebra homomorphism, then we are done. Verifying the map commutes with the multiplication is straightforward, we only verify the addition here. 
   Let $\frac{a_1 \otimes b_1}{c_1 \otimes d_1}$ and $\frac{a_2 \otimes b_2}{c_2 \otimes d_2}$ be elements in $( (A \otimes_A B) \backslash (\mathfrak{p}B \otimes_A B) )^{-1} (A \otimes_A B)$. Then
	\begin{align*}
	\psi_1 \left(\frac{a_1 \otimes b_1}{c_1 \otimes d_1}\right) +  \psi_1 \left(\frac{a_2 \otimes b_2}{c_2 \otimes d_2}\right)&=\frac{a_1}{c_1} \otimes \frac{b_1}{d_1} + \frac{a_2}{c_2} \otimes \frac{b_2}{d_2}\\
	&= \frac{a_1 c_2}{c_1 c_2} \otimes \frac{b_1 d_2}{d_1 d_2} + \frac{c_1 a_2}{c_1 c_2} \otimes \frac{d_1 b_2}{d_1 d_2} \\
	&=\left(\frac{1}{c_1 c_2} \otimes \frac{1}{d_1 d_2}\right) \cdot ( a_1 c_2 \otimes b_1 d_2 + c_1 a_2 \otimes d_1 b_2 ).
	\end{align*}
Again, 	
\begin{align*}
\psi_1 \left( \frac{a_1 \otimes b_1}{c_1 \otimes d_1} + \frac{a_2 \otimes b_2}{c_2 \otimes d_2} \right) &= \psi_1 \left( \frac{ (a_1 \otimes b_1) \cdot (c_2 \otimes d_2)}{ (c_1 \otimes d_1) \cdot (c_2 \otimes d_2) } + \frac{ (a_2 \otimes b_2) \cdot (c_1 \otimes d_1)}{ (c_2 \otimes d_2) \cdot (c_1 \otimes d_1) } \right)\\
& = \psi_1 \left( \frac{1}{c_1 c_2 \otimes d_1 d_2} \cdot (a_1 c_2 \otimes b_1 d_2 + c_1 a_2 \otimes d_1 b_2) \right) \\
&= \left(\frac{1}{c_1 c_2} \otimes \frac{1}{d_1 d_2}\right) \cdot \left( a_1 c_2 \otimes b_1 d_2 + c_1 a_2 \otimes d_1 b_2 \right).
\end{align*}
	\end{proof}
    
The following theorem determines the associated primes of  an extension of an ideal.

	\begin{theorem}\label{theorem: associated primes under extension from Matsumura}
		\cite[Theorem 23.2]{matsumura1989commutative} Let $\phi: A\to B$ be a homomorphism of Noetherian rings, and let $E$ be an $A$-module and $G$ a $B$-module. If $G$ is flat over $A$, then we have the following
		$$	\Ass_B (E \otimes_A G) = \bigcup_{\mathfrak{p} \in \Ass_A(E)} \Ass_B (G/\mathfrak{p}G).$$
	\end{theorem}

\begin{lemma}\label{proposition: associated primes of symbolic powers}
	Let $\phi: A\to B$ be a flat homomorphism. We set $\Ass_A^*(I)=\cup_{n\in \NN}\Ass_A(I^n)$.  If $I\subset A$ is a ideal then we have $\Ass_B(IB)=\cup_{\pp \in \Ass_A(I)} \Ass_B(\pp B) $. Furthermore, if $ \Ass_B^*(IB)=\{\pp B ~|~ \pp \in \Ass_A^*(I)
	\}$ then we have 
	\begin{enumerate}
\item $\Ass_B(IB)=\{\pp B~|~ \pp \in \Ass_A(I) \}$,
\item $\displaystyle \bigcup_{\qq\in\Ass_B(I^{(n)}B)}\qq \subset\bigcup_{\pp\in\Ass_A(I)}\pp B \text{ for all } n \geqslant 1.$ 
	\end{enumerate}
\end{lemma}
\begin{proof}
	Since  $\phi$ is a flat  then $A/I \otimes_A B \cong B/IB $. Now  by \Cref{theorem: associated primes under extension from Matsumura}, 
$$ \Ass_B(IB)=	\Ass_B (A/I\otimes_A B) = \bigcup_{\mathfrak{p} \in \Ass_A(A/I)} \Ass_B (B /\mathfrak{p}B)=\bigcup_{\pp \in \Ass(I)}\Ass_B(\pp B).$$
By the assumption $\pp B $ is prime for each $\pp \in \Ass_A(I)$. Hence we get  
 $$\Ass_B(IB) = \bigcup_{\pp \in \Ass(I)}\Ass_B(\pp B)=\{\pp B~|~ \pp \in \Ass(I) \}.$$
Note that by \cite[Lemma 2.2]{ha2023binomial} if $\pp \in \Ass(I^{(n)})$ then $\pp \in \Ass(I^n) $ and  $\pp \subset \qq $ for some $\qq\in \Ass(I)$. Thus we get $\displaystyle \cup_{\pp\in \Ass(I^{(n)})} \pp \subset \cup_{\pp\in \Ass(I)} \pp $ for all $n \geqslant 1$. Since $ \Ass^*(IB)=\{\pp B ~|~ \pp \in \Ass^*(I) \}$ then using \Cref{theorem: associated primes under extension from Matsumura} we get 
$$ \bigcup_{\qq\in\Ass(I^{(n)}B)}\qq = \bigcup_{\substack{\qq\in \Ass(\pp B),\\\pp \in \Ass\left(I^{(n)} \right) }}\qq = \bigcup_{\pp \in \Ass(I^{(n)})} \pp B  \subset \bigcup_{\pp\in\Ass(I)}\pp B, \text{ as } \bigcup_{\pp \in \Ass\left( I^{(n)} \right)} \pp  \subset \bigcup _{\pp \in \Ass(I)} \pp. $$

	\end{proof}


We are ready to prove the main theorem of this section.  The theorem states that if a flat morphism extends associative primes to associative primes and each associative prime is a contraction of some ideal, then symbolic powers of the extension of an ideal is the extension of the symbolic powers of the ideal. 

\begin{theorem} \label{theorem: symbolic powers under flat extension}Let $A$ and $B$ be Noetherian rings  and $\phi: A\to B$ be a flat morphism. Assume that $I\subset A$ is a proper ideal such that $\Ass^*(IB)=\{\pp B~|~ \text{ for all } \pp\in \Ass^*(I) \}$, and every associated prime of $I$ is a contraction of some ideal in $B$. Then $(IB)^{(n)} = I^{(n)}B $, for all $n \geqslant 1$.   Furthermore, if $I^l$ is a contraction of some ideals for all $\ell\geqslant 1$, then $I^{(m)} \subset I^r$, if and only if $ IB^{(m)} \subset IB^r$, for all $m,r \geqslant 1$, then we get $\rho(IB)=\rho(I)$ and $\rho_a(IB)=\rho(I)$. 
\end{theorem}
	\begin{proof} 
	First we check that as in Lemma \ref{proposition: associated primes of symbolic powers},  $\Ass_B (IB) = \{ {\pp}B ~|~ {\pp} \in \Ass_A (I) \} $. 	
Using the definition of symbolic powers of an ideal and the assumption, we get
		
$$(IB)^{(n)} = \bigcap_{{\pp} \in \Ass_B(IB)} ((IB)^n B_{{\pp}} \cap B)=\bigcap_{\{ {\pp}B ~|~ {\pp} \in \Ass_A(I) \}} ((IB)^n B_{ {\pp}B } \cap B).$$

Applying Lemma \ref{lemma: tensor of localization}, we have $B_{\mathfrak{p}B} \cong A_{\mathfrak{p}} \otimes_A B_{\mathfrak{p}B}$. Then by Lemma \ref{lemma: extension of an ideal}, we have the following:
			\begin{align*}
			(IB)^n B_{\mathfrak{p}B} \cap B &\cong (I^n \otimes_A B) (A_{\mathfrak{p}} \otimes_A B_{\mathfrak{p}B}) \cap (A \otimes_A B)
		\end{align*}
\par Next we want to show that $(( I^n A_{\mathfrak{p}} \cap A ) \otimes_A B_{\mathfrak{p}B}) \cap (A \otimes_A B) = ( I^n A_{\mathfrak{p}} \otimes_A B_{\mathfrak{p}B} ) \cap (A \otimes_A B)$. 

It is known that $A \otimes_A B \subset A \otimes_A B_{\mathfrak{p}B} $ if $S = B \backslash \mathfrak{p}B$ has no zero divisors. Suppose $S$ has no zero divisors, then applying a property from \cite[section 7]{matsumura1989commutative}, which states that if $M_2$ is flat over $M_1$ and $M_3$ is flat over $M_2$, then $M_3$ is flat over $M_1$. Here we consider $M_3$ as $B_{\mathfrak{p}B}$, $M_2$ as $B$, and $M_1$ as $A$, then we can see $B_{\mathfrak{p}B}$ is flat over $A$. By \cite[Theorem 7.4(I)]{matsumura1989commutative}, we consider $I^n A_{\mathfrak{p}}$ and $A$ as submodules of $A_{\mathfrak{p}}$ as $A$-modules, then $\left(I^n A_{\mathfrak{p}} \otimes_A B_{\mathfrak{p}B}\right) \cap \left( A \otimes_A B_{\mathfrak{p}B}\right) = ( I^n A_{\mathfrak{p}} \cap A ) \otimes_A B_{\mathfrak{p}B}$.  Therefore, intersecting it with $A\otimes_A B$ will not change the original intersection. 

If $S$ has zero divisors, then $B \backslash B_{\mathfrak{p}B}$ has exactly the same set of zero divisors in $S$, denoted as $S'$. Note that any $a \otimes b \in I^n A_{\mathfrak{p}} \otimes_A S'$ is always $0$ in $I^n A_{\mathfrak{p}} \otimes_A B_{\mathfrak{p}B}$, and similarly for $a \otimes b \in A \otimes_A S'$ in $A \otimes_A B_{\mathfrak{p}B}$. Observe that these 0's in $A \otimes_A B_{\mathfrak{p}B}$ are 0's in $( I^n A_{\mathfrak{p}} \otimes_A B_{\mathfrak{p}B} ) \cap (A \otimes_A B)$, so we can always intersect $I^n A_{\mathfrak{p}} \otimes_A B_{\mathfrak{p}B}$ with $A \otimes_A B_{\mathfrak{p}B}$ first, then $A \otimes_A B$, which leads to the same equality. Therefore $( I^n A_{\mathfrak{p}} \otimes_A B_{\mathfrak{p}B} ) \cap (A \otimes_A B) = (( I^n A_{\mathfrak{p}} \cap A ) \otimes_A B_{\mathfrak{p}B}) \cap (A \otimes_A B)$ as we desire. 

Applying lemma \ref{lemma: extension of an ideal} again on $(( I^n A_{\mathfrak{p}} \cap A ) \otimes_A B_{\mathfrak{p}B}) \cap (A \otimes_A B)$, we will have our second isomorphism $(( I^n A_{\mathfrak{p}} \cap A ) \otimes_A B_{\mathfrak{p}B}) \cap (A \otimes_A B) \cong (I^n A_p \cap A )B_{\pp B} \cap B$. Note that lemma \ref{lemma: extension of an ideal} is used previously for the ``backward" direction for the first isomorphism, and now applying it again for the ``forward" direction for our second isomorphism, so we will have a new equality (two isomorphisms give an identity map), which is $(IB)^n B_{\mathfrak{p}B} \cap B = (I^n A_p \cap A )B_{\pp B} \cap B$. Then the following equality is given:
\begin{align*}
(IB)^{(n)}  = \bigcap_{\pp\in \Ass(I)} (IB)^nB_{\pp B}\cap B = \bigcap_{\pp \in \Ass(I)} ((I^n A_p \cap A )B_{\pp B}) \cap B
\end{align*}

We want to prove the following labeled equalities:

\begin{align}
&\notag (IB)^{(n)}  = \bigcap_{\pp \in \Ass(I)} ((I^n A_p \cap A )B_{\pp B}) \cap B \\
 \label{eqn1}     &= \bigcap_{\pp \in \Ass(I)} I^{(n)}B_{\pp B} \cap B \\
  \label{eqn2}                & = \bigcap_{\pp \in \Ass(I)} I^{(n)}B B_{\pp B} \cap B \\
 \label{eqn3}     &= I^{(n)}B_W \cap B, \text{ where } W=B - \bigcup_{\pp\in \Ass(I)} \pp B \\
  \label{eqn4}   & =(I^{(n)}B)B_W\cap B \\
 \label{eqn5}       &= I^{(n)}B 
\end{align}

\vspace{5mm}

First we prove equality \ref{eqn1}.  Since $\phi$ is flat, and  extension commutes with intersection under flat morphism, we get 
\begin{align*}
	I^{(n)}B_{\pp B}= \left(\cap_{\pp^\prime\in \Ass(I)} (I^nA_\pp^{\prime}\cap A)\right) B_{\pp B}= \bigcap_{\pp^\prime\in \Ass(I)} ((I^n A_{\pp^\prime}\cap A) B_{\pp B})= (I^nA_\pp \cap A)B_{\pp B}.
\end{align*} 

The equality \ref{eqn2} is straight forward.

The proof of equality \ref{eqn3} is similar to \cite[Remark 2.8]{grifo2021symbolic}. We add the proof for completion. Note that 
$I^{(n)}B_W\cap B=\{ r \in B ~|~ rs\in I^{(n)} \text{ for some } s\in W \}$ and thus 
\begin{align*}
\displaystyle I^{(n)}B_W\cap B \subset \bigcap_{\pp \in \Ass(I)}\left( \left( I^{(n)}B_{\pp B} \right) \cap B \right).
\end{align*}

Now let $r \in \bigcap_{\pp \in \Ass(I)}\left( \left( I^{(n)}B_{\pp B} \right) \cap B \right) $. Then for each $\pp B$ there is $s\not\in \pp B$ such that $rs\in I^{(n)}B$. Thus $I^{(n)}B:r $ is not contained in any $\pp B \in \Ass(IB)$. Since the associated primes are finite then by Prime avoidance theorem $$I^{(n)}B:r \not \subset \bigcup_{\pp \in \Ass(I)} \pp B .$$ 

Then there is a $s \in W$ such that $sr \in I^{(n)}B$. Hence we have the equality as sets.

Again the equality \ref{eqn4} is straight forward. Lastly, we prove equality \ref{eqn5}. 
Now from \cite[Proposition 3.11(ii)]{atiyah1994introduction} $$(I^{(n)}B)B_W\cap B =\bigcup_{s\in W} I^{(n)}B : s .  $$

Now by Lemma \ref{proposition: associated primes of symbolic powers}, we have
\begin{align*}
\displaystyle \bigcup_{\Ass \qq \in\left( I^{(n)}B  \right)} \qq \subset\bigcup_{\qq\in \Ass (IB)}\qq=\bigcup_{\pp \in \Ass(I)}\pp B.
\end{align*}

Therefore whenever $s\in W$ then $s \not \in \qq$, where $\qq$ is an associated prime of $I^{(n)}B$. By \cite[Proposition 4.4(iii)]{atiyah1994introduction} we get 
   $$(I^{(n)}B)B_W\cap B =\bigcup_{s\in W} I^{(n)}B : s =I^{(n)}B.   $$
   
Note that $I^{(m)} \subset I^{(m)}B \cap A \subset IB^r \cap A=I^rB\cap A =I^r$, the last equality holds since $I^r$ is a contraction of some ideal. 
	\end{proof}
	
\begin{remark} Note that if placing $\Ass^*(IB)=\{\pp B~|~ \text{ for all } \pp\in \Ass^*(I) \}$ by $\Ass(IB^N)=\{\pp B~|~ \text{ for all } \pp\in \Ass(I^N) \}$, then the theorem \ref{theorem: symbolic powers under flat extension} is only true up for $ 1 \leqslant n \leqslant N$.

If one uses the minimal prime definition of the symbolic powers, then it is enough to assume that 
 $\Min(IB)=\{\pp B ~|~ \pp \in \Min (I) \}$.   
 \end{remark}
 
 The following example shows that the condition in Theorem \ref{theorem: symbolic powers under flat extension} that the associated primes are contraction of the ideals is necessary. 
 
 \begin{example}\label{example: contraction condition}
 	Consider the natural map $\phi: \ZZ[x] \to \QQ[x]$. Note that it is a flat map. Take $I=\langle x,2\rangle \subset \ZZ[x]$. Then $I$ is not a contraction of any ideals of $\QQ[x]$. Since $I$ is a maximal ideal, then $I^{(n)}=I^n$ for all $n\geqslant1$. But $I \QQ[x]=\langle x \rangle$. Hence $(I \QQ[x])^{(n)}=\langle x^n \rangle $ for all $n\geqslant 1$. But $I^{(n)}\QQ[x]=\langle x\rangle$ for all $n\geqslant 1$.  Thus $ (I \QQ[x])^{(n)} \neq I^{(n)}\QQ[x]$ for all $n\geqslant 2$. 
 \end{example}
 In \Cref{theorem: symbolic powers under flat extension} we extend \cite[Theorem 3.6]{GHMNmatroid}, \cite[Theorem 4]{Akessesh}, and \cite[proposition 2.1]{walker2018uniform} by replacing the \emph{faithfully flat} condition with \emph{flat} condition. Note that assumptions in \Cref{theorem: symbolic powers under flat extension} implies all these theorems and proposition, but not vice versa. We now give simple examples satisfying assumptions in \Cref{theorem: symbolic powers under flat extension}, but violating the \emph{faithfully flat} condition required for all theorems and proposition we just mentioned. 
 
 \begin{example}\label{example: flat not faithfully flat 1}
 	Consider  the map $\phi: \ZZ \to \ZZ\left[1 \slash p\right]$, where $p$ is a prime number, then this map is obviously \emph{flat}, since $I \otimes_{\ZZ} \ZZ\left[1 \slash p\right] \to \ZZ \otimes_{\ZZ} \ZZ\left[1 \slash p\right]$ is always injective for any ideal $I \subset \ZZ$. However, it is not \emph{faithfully flat} since $Spec(\ZZ\left[1 \slash p\right]) \to Spec(\ZZ)$ is not surjective for closed points, in particular, $(p)$ is the closed point. Finding any ideal $I$ in $\ZZ$ such that $(p)$ is not an associated prime of $I$, then $I^{(n)} \ZZ\left[1 \slash p\right] = (I \ZZ\left[1 \slash p\right])^{(n)}$. This example is also a good comparison to Example \ref{example: contraction condition}, since any primes in $\ZZ$ is a contraction of the same primes in $\ZZ\left[1 \slash p\right]$ except $(p)$. This example can be extended to PID and its prime ideal, since prime ideal (other than zero ideal) is also a maximal ideal, then it can be shown that map from PID to itself with reversing a prime ideal is \emph{flat} but not \emph{faithfully flat} by the same reason.
 \end{example}

\begin{example}\label{example: flat not faithfully flat 2}
  Let $A$ be a ring having at least two different prime ideals such that one is not contained in another one, and $B=A_\pp$ be localization of $A$ at an arbitrary prime $\mathfrak{p}$. It is a well-known fact that localization is an exact functor, so $B$ is a \emph{flat} $A$-module, and prime ideal $\mathfrak{p}$ in $A$ is the contraction of $\mathfrak{p}B$. Note that $\mathfrak{p}B$ is the only prime in $B$, then $\Spec(B) \rightarrow \Spec(A)$ is never surjective, so $B$ is not \emph{faithfully flat} as a $A$-module according to \cite[ex 3.16]{atiyah1994introduction}. 
\par 

One can take $A=\ZZ$ and $B$ as $\mathbb{Z}_{(p)}$ where $p$ is a prime number. Take  $I=(p^m)$ for some $m\geqslant 1$. Then $I$ satisfies assumptions in  \Cref{theorem: symbolic powers under flat extension}. Furthermore, letting $B$ be any localization of $A$ will meet the assumptions in \Cref{theorem: symbolic powers under flat extension}. This tells us that the stalk of $\Spec A$ gives us enough information of symbolic power of an ideal generated by the corresponding point in $\Spec A$. 
Another example would be take $A=\CC[t]$ and $B=\CC[t]_{(p(t))}$, where $p(t)$ is an irreducible polynomial in $\CC[t]$. 
	\end{example}	
	
	Applying our Theorem \ref{theorem: symbolic powers under flat extension}, we will have the following corollary, which is crucial for the main theorems in our section \ref{section: resurgence of sum of ideals}. We also explain the importance of this corollary in details at the beginning of the section \ref{section: resurgence of sum of ideals}.

 \begin{corollary}\label{corollary: resurgence via prime extension}
Let $\phi: A \to B$ be a flat morphism. If $\pp B$ is a prime ideal for all primes $\pp$ in $A$, then $\rho(I)=\rho(IB)$, and $\rho_a(I)=\rho_a(I)$ for any proper ideal $I$. 
 \end{corollary}
 \begin{proof}
The proof follows directly from Theorem \ref{theorem: symbolic powers under flat extension}. Note that if $\pp B$ is prime for all primes $\pp$ in $A$, then $\phi$ is faithfully flat and $IB\cap A=I$ for all ideals $I$ in $A$.  Indeed, by  \cite[theorem 7.2]{matsumura1989commutative}, if for every maximal ideal $\mathfrak{m}$ in $A$ we have $\mathfrak{m}B \neq B$, then $\phi$ is faithfully flat. 
\end{proof}
Morphisms above with the property that that extend every prime ideal as an prime ideal, also known as morphims with \emph{prime extension property}, see \cite{hochsterprime}. We will explain the geometric meaning of \emph{flat} and \emph{prime} extension, but before that, we need to recall the definition of \emph{fiber over points} first.

\begin{definition}		
	Let $f: X \to Y$ be a morphism of schemes, and let $y \in Y$ be a point. Let $\kappa(y)$ be the residue field of y, i.e. $\kappa(y) = \mathcal{O}_y / \mathfrak{m}_y$, and let $\Spec \, \kappa(y) \rightarrow Y$ be the natural morphism. Then we define the \emph{fiber of the morphism $f$ over the point $y$} to be the scheme 
	$$	X_y = X \times_Y \hspace{0.5mm} Spec \, \kappa(y).$$
	Note that \emph{fiber} $X_y$ is a scheme over $\kappa(y)$.
\end{definition}

Let  $f: \Spec B  \to \Spec A$ be a morphism of schemes, induced by ring homomorphism $\phi: A \to B$. Then the fiber over $\mathfrak{p} \in \Spec A$ is given by $\Spec  B_{\mathfrak{p}} = \Spec \left(B \otimes_A \kappa(\mathfrak{p})\right) = \Spec \left(B \otimes_A A_{\mathfrak{p}}/ \mathfrak{p} A_{\mathfrak{p}}\right)$. We have the following proposition:

\begin{proposition}
	If  the ring homomorphism $\phi: A \to B$ is flat then $B_{\mathfrak{p}}$ is an integral domain if only if $\mathfrak{p}B$ prime.
\end{proposition}

\begin{proof}
	Since $\phi$ is flat, we can consider $\mathfrak{p} \otimes_A B$ as $\mathfrak{p}B$. It's easy to see the followings:
	$$B_{\mathfrak{p}} \text{ is an integral domain } \iff B/\mathfrak{p}B \text{ is an integral domain }\iff \mathfrak{p}B \text{ is prime}.$$
\end{proof}

This proposition implies that if $A$ is reduced, and $\phi$ has \emph{prime extension property}, then $B$ is also reduced. Indeed, if $A$ contains the zero ideal as a prime ideal, so does $B$. If $\phi$ is a prime extension, then every $B_{\mathfrak{p}}$ is an integral domain. We are ready to give a geometric example if we recall one more lemma.

\begin{lemma}\label{lemma Lemma 10.39.5 from Stack}
	\cite[Lemma 10.39.5]{stacks-project} Let $M$ be an $R$-module, then $M$ be a flat $R$-module if only if for every ideal $I \subset R$ the map $I \otimes_R M \rightarrow R \otimes_R M$ is injective.
\end{lemma}

\begin{example}
	Let $\kk$ be an algebraically closed field, and let 
	$$	A =  \kk [t]  \text{ and }B = \kk [t,x,y]/(y - x^2 - tx).$$
	
	Let $\phi$ to be the ring homomorphism from $A$ to $B$ by sending $t$ to $t$, and let $\Phi$ be the morphism from $\Spec B$ to $\Spec A$ induced by $\phi$. Note that $A\otimes_AB$ is a domain. It is not hard to see $\phi$ is a \textit{flat} morphism. The map from $I \otimes_A B$ to $A \otimes_A B$ is always injective for every ideal $I \subset A$, then by Lemma \ref{lemma Lemma 10.39.5 from Stack} we get that $\phi$ is flat. The injectivity comes from the fact that both $y - x^2$ and $y - x^2 - tx$ are irreducible in $\kk[t,x,y]$.

	We identify closed points of $A$ with elements of $\kk$, and obviously the $(t - c)$ in $\kk [t,x,y]/(y - x^2 - tx)$ is irreducible for every $c \in \kk$, so $\phi$ is also a prime extension. The fiber $\Spec B_{(t-c)}$ is given by the plane curve $y = x^2 + cx$ in $\mathbb{A}_\kk^2$, which is an irreducible variety, and those plane curves are even homeomorphic to each other. 
	
	The situation doesn't always perform so nicely. If we consider $A$ as $\kk[t,s]$ and $B$ as $\kk[t,s,x,y] / (y - x^2 -tx)$, we also have $(t - c)$ as our point, then $B \otimes_A A_{\mathfrak{p}} / \mathfrak{p} A_{\mathfrak{p}}$ will be $\kk[s,s^{-1},x,y] / (y - x^2 -cx) = \kk(s)[x,y] / (y - x^2 -cx)$. However we have $(t-a, s-b)$ as our point where $a$,$b \in \kk$, now the corresponding fiber is $\Spec \kk[x,y] / (y - x^2 - ax)$, which is not homeomorphic to the former.
\end{example}

	\section{Resurgence of sum of ideals} \label{section: resurgence of sum of ideals}
	
	First note that in this section, we assume that all rings and algebras are \emph{regular}. We summarize current results of \emph{resurgence number} of sums of ideals for polynomial rings over a field, our goal is to prove these results for sums of finitely generated graded {$\kk$-algebras} which are integral domains, where $\kk$ is algebraically closed. In \cite[Theorem 2.6]{fiberprojective}, it has been proved that if 
	$I$ and $J$ are ideals in different polynomial rings over a same field, then
$$	\rho_a(I + J) = \overline{\rho}_{\alpha}(I + J) = \overline{\rho}(I + J) = {\max}\{ \rho_a(I),  \rho_a(J) \} .$$ 
	In \cite[Theorem 2.3]{sharpbound}, it has been proved that 
	 $$
	\max \{ \rho(I), \, \rho(J) \} \leqslant \rho(I + J) \leqslant \max \left\{ \rho(I), \, \rho(J), \, \frac{2 (\rho(I) + \rho(J))}{3} \right\}.
	$$

In our paper, \cite[Theorem 2.6]{fiberprojective} is generalized to Theorem \ref{theorem: asymptotic resurgence of sum of ideals}, and \cite[Theorem 2.3]{sharpbound} is generalized to Theorem \ref{theorem: resurgence bound of sum of ideals}. The generalization is not trivial. First, it is obvious that the maps from two polynomials rings $A$ and $B$ over a field $\kk$ to $A \otimes_{\kk} B$ preserve resurgence numbers. However, for our generalization, we need to apply Corollary \ref{corollary: resurgence via prime extension} to achieve that resurgence number of an ideal is same as resurgence number of its extension. Second, in the proof of Lemma \ref{lemma: a-resurgence of  I +J > a -sup resurgence}, we use vector space method, while Gröbner basis method was used in the polynomial case but could not be used in $\kk$-algebra case.

In  Lemma \ref{lemma: flat and prime under tensor product}, we will specify that under proper conditions resurgence of the extension is the same as the resurgence of the ideals.

		\begin{lemma}\label{lemma: flat and prime under tensor product}
		Let $A$ and $B$ be finitely generated $\kk$ algebras, where $\kk$ is algebraically closed, which are integral domains. If  $\phi: A \to R$ is  a canonical ring homomorphism, where $R=A\otimes_{\kk} B$, then $\phi$ is {flat} and has {prime extension} property. Furthermore, $ \rho(IR) =\rho(I)$ and $\rho_a(IR)=\rho_a(I)$ for any proper ideal $I$ from $A$ (or $B$).
	\end{lemma}
	
	\begin{proof}
	Recall that the canonical ring homomorphism is defined by
	$$
			\phi : A \to A \otimes_\kk B \text{ defined by } a \rightarrow a \otimes_\kk 1. 
		$$
		
		Observe that if we can show $\phi$ is injective, then $I \rightarrow A \otimes_A R$ is injective for every ideal $I \subset A$. Hence by Lemma \ref{lemma Lemma 10.39.5 from Stack},  $\phi$ is {flat}. Injectivity of $\phi$ is obvious since both $A$ and $B$ are  fintely generated algebras over an algebraically closed field $\kk$, and {integral domains}. If there exists a nonzero element in $\kk$ annihilating some elements in $A$ (or in $B$), then this element has to be zero in $A$ (or in $B$), which is a contradiction.

		Let $\mathfrak{p} \subset A$ be a prime ideal, we need to prove that $\pp R$ is also prime. Since $\phi$ is injective then $\mathfrak{p}R = \mathfrak{p} \otimes_\kk B$. Observe that $ R / \mathfrak{p}R = \left(A / \mathfrak{p}\right) \otimes_\kk B$ because $R$ is an {integral domain}.  
		
  Now $\left(A / \mathfrak{p}\right) \otimes_k B$ is an integral domain	by \cite[Lemma 1.54]{algvar}. 
	 It is a well-known fact that an ideal is prime if only if the quotient ring is an integral domain, so $\mathfrak{p}R$ is prime, and thus, $\phi$ has {prime extension} property. The last equality follows from  Corollary \ref{corollary: resurgence via prime extension}. 
\end{proof}	

The next lemma also extends  \cite[Lemma 2.4]{fiberprojective}. The proof of \cite[Lemma 2.4]{fiberprojective} depends on divisibility in the polynomial ring. While working with general domains, we treated them as vector spaces and used properties of basis. Let us recall the following definition from \cite{guardo2013asymptotic}.  
\begin{definition}
    Let $I$ be an ideal in a regular ring. Then $\lim_{t\to \infty} \sup \rho(I,t) $ is defined as follows: $$\rho(I,t)=\sup \left\{\dfrac{m}{r}~|~I^{(m)} \not\subseteq I^r, m \geqslant t, r \geqslant t\right\}, \text{ and }\rho^{\lim\sup}_a(I)=\lim_{t\to \infty} \sup \rho(I,t). $$
\end{definition}

	\begin{lemma} \label{lemma: a-resurgence of  I +J > a -sup resurgence}
		Let $A$ and $B$ be finitely generated $\kk$-algebras, and let $I \subset A$ and $J \subset B$ be nonzero proper homogeneous ideals. If  $I + J$ be the sum of ideal extensions of $I$ and $J$ in $R = A \otimes_\kk B$ then
$$
	\rho_a^{\lim \sup}(I + J) \geqslant {\max}  \left\{ \rho_a^{{\lim \sup}}(I), \, \rho_a^{{\lim \sup}}(J) \right\}. 
$$
	\end{lemma}
	\begin{proof}
		Take $\theta \in \QQ$ such that $\theta>\rho_a^{\lim \sup}(I + J) $.  Then there is a $s_0\in \NN$ such that for all $t\geqslant s_0, \theta >\rho(I+J,t)$, where $\rho(I+J,t)=\sup \left\{\dfrac{m}{r} ~|~I^{(m)}\not\subset I^r, m \geqslant t, r \geqslant t \right\}$.  Thus for any $h,r \in \NN$ so that $h,r\geqslant t\geqslant s_0$ and $\dfrac{h}{r}\geqslant \theta$, one has $(I+J)^r \supset (I+J)^{(h)} \supset I^{(h)} $. 
		Up to this  point  the proof idea is same as \cite[Lemma 2.4]{fiberprojective}. Now we need to show  that $I^{(h)} \subseteq I^r$.
		
      Let $V$ be a $\kk$ basis of $I$, which extends to $V^*$ a $\kk$-basis of $A$ containing $1$. Let $W$ be a $\kk$-basis of $J$, which extends to $W^*$, a $\kk$-basis of $B$ containing $1$. Then $V\otimes_{\kk} W^*$ is a $\kk$-basis for $I\otimes_{\kk} B$ and $V^*\otimes_{\kk} W$ is a $\kk$-basis for $V\otimes_{\kk} W^*$. Also, $V^*\otimes_{\kk} W^*$ is a $\kk$-basis of $R=A\otimes_{\kk} B$.
       Let $f\in I^{(h)}$ then $f=\sum_{i=1}^n c_iv_i $, where $c_i\in \kk\setminus \{0\}$ and $v_i\in V$.
       Now $$f\otimes 1 =\sum_{i=1}^ncv_i\otimes 1 \in (I+J)^{(h)}\subset (I+J)^r=\sum_{j=0}^r (IR)^i(JR)^{r-j}= \sum_{j=0}^r \left((IR)^j\cap(JR)^{r-j}\right).$$ The last equality follows from \cite[Lemma 3.1]{ha2020symbolic}. 
		Let $f \in I^{(h)} \subseteq A$ be an arbitrary element, and consider $f \otimes 1 \in (I + J)^{(h)} \subseteq (I + J)^r = \sum_{i=0}^{r} (IR)^i (JR)^{r-i}$. Since $1\notin W  $ then $v\otimes 1 \notin A\otimes J$. Thus none of the $c_iv_i\otimes1$ appearing in the expression of $f\otimes 1$ will be in $ (IR)^j\cap(JR)^{r-j} $, where $1\leqslant j \leqslant r-1$. Hence $f\otimes1 = \sum_{i=1}^ncv_i\otimes 1 \in (IR)^h$. 
		Now $(IR)^h \cong I^h \otimes B$ and thus $f\otimes 1 \in (IR)^h $ implies that $f\in I^h$. This completes the proof. 
\end{proof}

The following generalizes \cite[Lemma 2.5]{fiberprojective} to finitely generated algebra domains.  The proof of  \cite[Lemma 2.5]{fiberprojective} depends on the binomial summation formula (\cite[Theorem 3.4]{bionomial}) and numerical inequalities, which holds in Noetherian algebras over a field as well, so we omit the proof. 

\begin{lemma}\label{lemma: a-resurgence of I+J < a sup resurgence}
Let $A$ and $B$ be finitely generated algebras over an algebraically closed field $\kk$ which are integral domains, and let $I \subset A$ and $J \subset B$ be nonzero proper homogeneous ideals.  If $I + J$  denotes $IR+JR$, where $R=A\otimes_\kk B$, then   $\rho_a(I+J)\leqslant \max\left\{\rho_a^{\lim \sup }(I ), \rho_a^{\lim \sup }(J) \right\} $. 
\end{lemma}

Now we give the formula for the asymptotic resurgence of the sum of ideals in finitely generated $\kk$-algebra domians where $\kk$ is an algebraically closed field. 

\begin{theorem}\label{theorem: asymptotic resurgence of sum of ideals}
		Let $A$ and $B$ be finitely generated algebras over an algebraically closed field $\kk$ which are integral domains, and let $I \subset A$ and $J \subset B$ be nonzero proper homogeneous ideals.  If $I + J$  denotes $IR+JR$, where $R=A\otimes_\kk B$, then  
		$$	\rho_a(I + J) = \overline{\rho}_{\alpha}(I + J) = \overline{\rho}(I + J) = {\max}\{ \rho_a(I),  \rho_a(J) \}=\max\{\rho_a(IR),  \rho_a(JR)  \}.$$

	\end{theorem}
\begin{proof}
	The proof  has the same spirit as  \cite[Theorem 2.6]{fiberprojective} and it follows from  Lemma \ref{lemma: flat and prime under tensor product}, Lemma \ref{lemma: a-resurgence of  I +J > a -sup resurgence} and Lemma \ref{lemma: a-resurgence of I+J < a sup resurgence}. 
\end{proof}	
 
	\begin{theorem} \label{theorem: resurgence bound of sum of ideals} 
Let $A$ and $B$ be finitely generated algebras over an algebraically closed field $\kk$ which are integral domains, and let $I \subset A$ and $J \subset B$ be nonzero proper homogeneous ideals. Set $I+J=IR+JR$, where $R=A\otimes_{\kk} B$. 
Then $$
\max \{ \rho(I), \, \rho(J) \} \leqslant \rho(I + J) \leqslant \max \left\{ \rho(I), \, \rho(J), \, \frac{2 (\rho(I) + \rho(J))}{3} \right\}.
$$

Furthermore, if $\max \{ \rho(I), \, \rho(J) \} \geqslant 2 \min \{ \rho(I), \, \rho(J) \}$, then $\rho(I + J) = \max \{ \rho(I), \, \rho(J) \}$. 
	\end{theorem}
\begin{proof}
We give a sketch of the proof as it follows the same pattern as \cite[Theorem 2.3]{sharpbound}.	
First note that by Lemma \ref{lemma: flat and prime under tensor product}, $\rho(IR)=\rho(I)$ and $\rho(JR)=\rho(J)$. The binomial summation formula for symbolic powers hold for ideals in noetherian $\kk$-algebras (see \cite[Theorem 4.1]{ha2023binomial}. The other numerical inequalities in \cite[Theorem 2.3]{sharpbound} follows straightforward. Hence the required result follows. 
\end{proof}

\begin{remark}	
	To see why we need these conditions on $A$ and $B$ from Theorem \ref{theorem: asymptotic resurgence of sum of ideals}, first note that, if $A$ (as a field extension over $\kk$) is not finitely generated as a $\kk$-algebra, then $A \otimes_\kk A$ may not a Noetherian ring by \cite[Theorem 11]{vmos}. For an example, $\mathbb{R} \otimes_{\mathbb{Q}} \mathbb{R}$ is not a Notherian ring. Therefore, the condition being finitely generated is necessary. We also need $A$ and $B$ to be integral domains as $\kk$-algebras where $\kk$ is algebraically closed because of \cite[Lemma 1.54]{algvar}, this will make $A \otimes_\kk B$ to be an integral domain. A famous example for violating the algebraically closed condition, then we do not have integral domain, is $\mathbb{C} \otimes_{\mathbb{R}} \mathbb{C}$. Moreover, note by \cite[Theorem 6(e)]{Regulartensor}, $A\otimes_{\kk} B$ is regular if $A$ and $B$ are regular. 
\end{remark}

\bibliographystyle{alpha}
	\bibliography{references}

@article{Regulartensor,
	title={Tensor products of some special rings},
	author={Tousi, Masoud and Yassemi, Siamak},
	journal={Journal of Algebra},
	volume={268},
	number={2},
	pages={672--676},
	year={2003},
	publisher={Elsevier}
}

@article{ha2023binomial,
	title={Binomial expansion for saturated and symbolic powers of sums of ideals},
	author={H{\`a}, Huy T{\`a}i and Jayanthan, AV and Kumar, Arvind and Nguyen, Hop D},
	journal={Journal of Algebra},
	volume={620},
	pages={690--710},
	year={2023},
	publisher={Elsevier}
}

@incollection {hochsterprime,
	AUTHOR = {Hochster, Melvin and Jeffries, Jack},
	TITLE = {Extensions of primes, flatness, and intersection flatness},
	BOOKTITLE = {Commutative algebra---150 years with {R}oger and {S}ylvia
	{W}iegand},
	SERIES = {Contemp. Math.},
	VOLUME = {773},
	PAGES = {63--81},
	PUBLISHER = {Amer. Math. Soc., [Providence], RI},
	YEAR = {[2021] \copyright 2021},
	ISBN = {978-1-4704-5601-6},
	MRCLASS = {13B02 (13B10 13B40 13C11)},
	MRNUMBER = {4321391},
	MRREVIEWER = {Marco\ Fontana},
	DOI = {10.1090/conm/773/15533},
	URL = {https://doi.org/10.1090/conm/773/15533},
}

@article{grifo2021symbolic,
	title={Symbolic powers},
	author={Grifo, Elo{\'\i}sa},
	journal={Boletim do CIM},
	year={2021},
	publisher={Centro Internacional de Matem{\'a}tica}
}

@article {GHMNmatroid,
	AUTHOR = {Geramita, A. V. and Harbourne, B. and Migliore, J. and Nagel,
	U.},
	TITLE = {Matroid configurations and symbolic powers of their ideals},
	JOURNAL = {Trans. Amer. Math. Soc.},
	FJOURNAL = {Transactions of the American Mathematical Society},
	VOLUME = {369},
	YEAR = {2017},
	NUMBER = {10},
	PAGES = {7049--7066},
	ISSN = {0002-9947,1088-6850},
	MRCLASS = {14N20 (05B35 05E40 13C40 13D02 13F55 14M05)},
	MRNUMBER = {3683102},
	MRREVIEWER = {L\^e\ Tu\^an\ Hoa},
	DOI = {10.1090/tran/6874},
	URL = {https://doi.org/10.1090/tran/6874},
}

@article {Akessesh,
	AUTHOR = {Akesseh, Solomon},
	TITLE = {Ideal containments under flat extensions},
	JOURNAL = {J. Algebra},
	FJOURNAL = {Journal of Algebra},
	VOLUME = {492},
	YEAR = {2017},
	PAGES = {44--51},
	ISSN = {0021-8693,1090-266X},
	MRCLASS = {13F20 (14N05 14N20)},
	MRNUMBER = {3709140},
	MRREVIEWER = {Krishna\ Hanumanthu},
	DOI = {10.1016/j.jalgebra.2017.07.026},
	URL = {https://doi.org/10.1016/j.jalgebra.2017.07.026},
}

@article {MaScwede,
	AUTHOR = {Ma, Linquan and Schwede, Karl},
	TITLE = {Perfectoid multiplier/test ideals in regular rings and bounds
	on symbolic powers},
	JOURNAL = {Invent. Math.},
	FJOURNAL = {Inventiones Mathematicae},
	VOLUME = {214},
	YEAR = {2018},
	NUMBER = {2},
	PAGES = {913--955},
	ISSN = {0020-9910,1432-1297},
	MRCLASS = {13A35 (14F18)},
	MRNUMBER = {3867632},
	MRREVIEWER = {Ana\ Bravo},
	DOI = {10.1007/s00222-018-0813-1},
	URL = {https://doi.org/10.1007/s00222-018-0813-1},
}

@article {HochsterHuneke,
	AUTHOR = {Hochster, Melvin and Huneke, Craig},
	TITLE = {Comparison of symbolic and ordinary powers of ideals},
	JOURNAL = {Invent. Math.},
	FJOURNAL = {Inventiones Mathematicae},
	VOLUME = {147},
	YEAR = {2002},
	NUMBER = {2},
	PAGES = {349--369},
	ISSN = {0020-9910,1432-1297},
	MRCLASS = {13A10 (13H05)},
	MRNUMBER = {1881923},
	MRREVIEWER = {Irena\ Swanson},
	DOI = {10.1007/s002220100176},
	URL = {https://doi.org/10.1007/s002220100176},
}

@article {EinLazarsfeldSmith,
	AUTHOR = {Ein, Lawrence and Lazarsfeld, Robert and Smith, Karen E.},
	TITLE = {Uniform bounds and symbolic powers on smooth varieties},
	JOURNAL = {Invent. Math.},
	FJOURNAL = {Inventiones Mathematicae},
	VOLUME = {144},
	YEAR = {2001},
	NUMBER = {2},
	PAGES = {241--252},
	ISSN = {0020-9910,1432-1297},
	MRCLASS = {13A10 (13H05 14Q20)},
	MRNUMBER = {1826369},
	MRREVIEWER = {Irena\ Swanson},
	DOI = {10.1007/s002220100121},
	URL = {https://doi.org/10.1007/s002220100121},
}

@article{guardo2013asymptotic,
	title={Asymptotic resurgences for ideals of positive dimensional subschemes of projective space},
	author={Guardo, Elena and Harbourne, Brian and Van Tuyl, Adam},
	journal={Advances in Mathematics},
	volume={246},
	pages={114--127},
	year={2013},
	publisher={Elsevier}
}

@article{bocci2010comparing,
	title={Comparing powers and symbolic powers of ideals},
	author={Bocci, Cristiano and Harbourne, Brian},
	journal={Journal of Algebraic Geometry},
	volume={19},
	number={3},
	pages={399--417},
	year={2010}
}

@article{ha2020symbolic,
	title={Symbolic powers of sums of ideals},
	author={H{\`a}, Huy T{\`a}i and Nguyen, Hop Dang and Trung, Ngo Viet and Trung, Tran Nam},
	journal={Mathematische Zeitschrift},
	volume={294},
	number={3},
	pages={1499--1520},
	year={2020},
	publisher={Springer}
}

@article{walker2018uniform,
	title={Uniform Harbourne--Huneke bounds via flat extensions},
	author={Walker, Robert M},
	journal={Journal of Algebra},
	volume={516},
	pages={125--148},
	year={2018},
	publisher={Elsevier}
}

@article{matsumura1989commutative,
  title={Commutative Ring Theory},
  author={Matsumura, Hideyuki},
  isbn={9780521367646},
  lccn={86011691},
  series={Cambridge Studies in Advanced Mathematics},
  url={https://books.google.com/books?id=yJwNrABugDEC},
  year={1989},
  publisher={Cambridge University Press}
}

@article{atiyah1994introduction,
  title={Introduction To Commutative Algebra},
  author={Atiyah, M.F. and MacDonald, I.G.},
  isbn={9780813345444},
  series={Addison-Wesley series in mathematics},
  url={https://books.google.com/books?id=HOASFid4x18C},
  year={1994},
  publisher={Avalon Publishing}
}

@misc{stacks-project,
  author       = {The {Stacks project authors}},
  title        = {The Stacks project},
  howpublished = {\url{https://stacks.math.columbia.edu}},
  year         = {2024},
}

@article{fiberprojective,
title = {Resurgence numbers of fiber products of projective schemes},
journal = {Collectanea Mathematica},
volume = {72},
pages = {605–614},
year = {2021},
issn = {2038-4815},
doi = {https://doi.org/10.1007/s13348-020-00302-5},
url = {https://link.springer.com/article/10.1007/s13348-020-00302-5},
author = {Bisui, Sankhaneel and Hà, Huy Tài and Jayanthan, A. V. and Abu Chackalamannil Thomas}
}

@article{sharpbound,
title = {A sharp bound for the resurgence of sums of ideals},
journal = {Proceedings of the American Mathematical Society},
volume = {152},
pages = {1405-1418},
year = {2024},
issn = {1088-6826},
doi = {https://doi.org/10.1090/proc/16655},
url = {https://www.ams.org/journals/proc/2024-152-04/S0002-9939-2024-16655-7},
author = {Kien, Do Van and Nguyen, Hop Dang and Thuan, Le Minh}
}

@article{vmos,
title = {On the minimal prime ideals of a tensor product of two fields},
journal = {Mathematical Proceedings of the Cambridge Philosophical Society},
volume = {84},
pages = {25-35},
year = {1978},
doi = {https://doi.org/10.1017/S0305004100054840},
url = {https://www.cambridge.org/core/journals/mathematical-proceedings-of-the-cambridge-philosophical-society/article/on-the-minimal-prime-ideals-of-a-tensor-product-of-two-fields/7DA12CE1CB1E4284405A9BE50E3303EB},
author = {Peter V{\'a}mos}
}

@book{algvar,
  title={Algebraic Geometry: An Introduction to Birational Geometry of Algebraic Varieties},
  author={Shigeru Iitaka},
  isbn={1461381215},
  series={Graduate Texts in Mathematics},
  url={https://link.springer.com/book/9781461381211},
  year={1982},
  publisher={Springer New York, NY}
}

@article{bionomial,
title={Symbolic powers of sums of ideals},
author = {Ha, Tai and Nguyen, Dang and Ngo, Trung and Trung, Trân},
year = {2020},
month = {04},
title = {Symbolic powers of sums of ideals},
volume = {294},
journal = {Mathematische Zeitschrift}
}
		
\end{document}